\documentclass{svproc}

\usepackage[utf8]{inputenc}
\usepackage[T1]{fontenc}
\usepackage{url}

\usepackage{graphicx}
\usepackage{amsfonts}
\usepackage{algorithm}
\usepackage{algpseudocode}
\usepackage{amsmath}
\usepackage{booktabs}
\usepackage{array}

\newcommand{\R}{\mathbb{R}}

\makeatletter
\@ifundefined{theorem}{\spnewtheorem{theorem}{Theorem}{\bfseries}{\itshape}}{}
\@ifundefined{definition}{\spnewtheorem{definition}{Definition}{\bfseries}{\rmfamily}}{}
\@ifundefined{remark}{\spnewtheorem{remark}{Remark}{\bfseries}{\rmfamily}}{}
\@ifundefined{example}{\spnewtheorem{example}{Example}{\bfseries}{\rmfamily}}{}
\makeatother

\begin{document}
	
	\mainmatter
	
	\title{Riemannian Regression}
	
	\titlerunning{Riemannian Regression}
	
	\author{Oldemar Rodríguez\inst{1}}
	
	\authorrunning{Oldemar Rodríguez}
	
	\institute{School of Mathematics and Research Center for Pure and Applied Mathematics (CIMPA), University of Costa Rica, San José, Costa Rica,\\
		\email{oldemar.rodriguez@ucr.ac.cr},\\ WWW home page:
		\texttt{https://oldemarrodriguez.com/}}
	
	\maketitle
	
	\begin{abstract}
		Classical linear regression assumes that the relevant geometry of the predictor space is Euclidean and that all centered observations contribute to the least-squares fit in the same geometric scale. This paper proposes \emph{Riemannian Regression}, a regression framework in which the usual vector differences are replaced by locally weighted differences induced by a data-dependent similarity structure. We introduce a generalized framework, termed {\em Riemannian Regression}, extending classic regression to any data endowed with a local distance structure. By equipping data tables with local metrics, we adapt regression model to incorporate manifold geometry.
		Given a similarity matrix $S=(S_{ij})$, obtained from UMAP, ISOMAP, or DBSCAN \cite{mcinnes,isomap,dbscan}, we define the dissimilarity coefficient $\rho_{ij}=1-S_{ij}$ and the induced subtraction
		$
		x_i\ominus x_j=\rho_{ij}(x_i-x_j).
		$
		A Riemannian center $g=x_\lambda$ is selected as a discrete Fr\'{e}chet mean, and regression is performed on the Riemannian-centered variables $X_R=W X_{c,\lambda}$ and $y_R=W y_{c,\lambda}$, where $W=\operatorname{diag}(\rho_{1\lambda},\ldots,\rho_{n\lambda})$. The resulting estimator has the weighted least-squares form
		$
		\widehat\beta_R=(X_{c,\lambda}^{t}W^2X_{c,\lambda})^{-1}X_{c,\lambda}^{t}W^2y_{c,\lambda}.
		$
		The proposed approach preserves the linear form of the regression model while changing the geometry of the fit. The paper develops three ways to construct the local metric: UMAP-based fuzzy similarities, ISOMAP-based normalized geodesic distances, and DBSCAN-based density similarities. Simulated examples and the Abalone data set illustrate how Riemannian Regression can reduce the influence of locally anomalous observations and adapt to regions with different local densities.
		\keywords{Riemannian regression, UMAP, ISOMAP, DBSCAN, local metrics, weighted least squares, Riemannian statistics}
	\end{abstract}
	
	\section{Introduction}
	
	Classical multiple linear regression models a response variable $y\in\R^n$ from a predictor matrix $X\in\R^{n\times p}$ by assuming that the relevant differences between observations are Euclidean. If $\bar{x}$ and $\bar{y}$ denote the usual arithmetic means, the centered least-squares estimator is
	$
	\widehat\beta_c=(X_c^tX_c)^{-1}X_c^t y_c,
	$
	where $X_c$ has rows $x_i-\bar{x}$ and $y_c$ has entries $y_i-\bar{y}$.
	
	This approach is effective when the global Euclidean geometry adequately represents the structure of the data. However, many data sets contain nonlinear manifolds, local density differences, clusters, disconnected regions, or points that are close in Euclidean coordinates but far in the intrinsic geometry of the data. In those cases, treating all centered differences equally may distort the regression fit.
	
	The objective of this paper is to extend the Riemannian statistical framework previously used for Riemannian Principal Component Analysis in \cite{Rodriguez2026PCA} to linear regression. The central idea is to keep the linear regression model, but replace the Euclidean centering and subtraction by a Riemannian centering based on local similarities. Given a similarity matrix $S=(S_{ij})$, we define
	$
	\rho_{ij}=1-S_{ij},
	$
	and use $\rho_{ij}$ to weight differences between observations. Thus, the Riemannian subtraction is
	$
	x_i\ominus x_j=\rho_{ij}(x_i-x_j).
	$
	This makes observations that are highly similar locally contribute less to the difference, while observations that are far apart in the induced geometry retain larger differences.
	
	This paper develops a unified framework in which $S$ can be obtained from three different sources:
	\begin{enumerate}
		\item UMAP, through fuzzy simplicial similarities \cite{mcinnes};
		\item ISOMAP, through normalized geodesic distances on a nearest-neighbor graph \cite{isomap};
		\item DBSCAN, through density, neighborhood, cluster, and noise information \cite{dbscan}.
	\end{enumerate}
	All three methods produce a similarity matrix $S$ and then follow the common pipeline
	$
	S \longrightarrow \rho=1-S \longrightarrow x_i\ominus x_j=\rho_{ij}(x_i-x_j).
	$

	\section{From Euclidean Data Tables to Discrete Riemannian Manifolds}
	\label{sec:discrete_riemannian_manifold}
	
	The construction used in this paper is motivated by the classical problem of analyzing data on a Riemannian manifold. Let $M$ be a manifold and let $H$ be a geodesic submanifold of $M$. The projection of a point $x\in M$ onto $H$ is defined as the point on $H$ that is nearest to $x$ in geodesic distance. Thus, the projection operator $\pi_H:M\to H$ is given by
	\begin{equation}
		\label{eq:projection_geodesic_submanifold}
		\pi_H(x)
		=
		\arg\min_{y\in H} d_M(x,y)^2.
	\end{equation}
	Since this projection is defined through a minimization problem, there is no general guarantee that the projection exists or that it is unique. However, by restricting the analysis to a sufficiently small neighborhood around the mean, the projection onto a geodesic submanifold through the mean is unique and can be approximated linearly in the tangent space of $M$. This is the idea behind Principal Geodesic Analysis (PGA) \cite{flet,pennec}: we approximates the projection in the tangent space and then searches for nested geodesic submanifolds that maximize the projected variance of the data; see \cite{flet} for details.
	
	Following the standard tangent-space treatment of Riemannian data \cite{pennec}, more explicitly, given manifold-valued observations
	$
	x_1,x_2,\ldots,x_n\in M,
	$
	we first compute the intrinsic mean
	$$
	\mu
	=
	\arg\min_{x\in M}
	\sum_{i=1}^{n}d_M(x,x_i)^2.
	$$
	Then each observation is represented in the tangent space at the mean. In the usual tangent-space approximation, this is written as
	$
	u_i=x_i-\mu,
	$
	with the understanding that, on a true manifold, this expression represents the local tangent vector associated with the displacement from $\mu$ to $x_i$. The covariance matrix is then approximated by
	$$
	C=
	\frac{1}{n}
	\sum_{i=1}^{n}
	u_i
	u_i^t,
	$$
	and its eigendecomposition gives principal directions and variances.
	
	The objective of the present work is to generalize this idea to an arbitrary Euclidean data table, even when no smooth manifold is explicitly given. Let
	$$
	X=\{x_1,x_2,\ldots,x_n\}\subset\mathbb{R}^p
	$$
	be a data set embedded in a high-dimensional Euclidean space. The key step is to replace the unknown smooth manifold by a \emph{local manifold approximation}. For each data point $x_i$, we define a local neighborhood using a base metric $d$. The $k$ nearest neighbors of $x_i$ are denoted by
	$$
	\mathcal{N}_k(x_i)
	=
	\left\{
	x_j\in X\setminus\{x_i\}:
	\begin{array}{l}
		d(x_i,x_j)\leq d(x_i,x_\ell)\\
		\text{for points }x_\ell\text{ outside the selected }k\text{ neighbors}
	\end{array}
	\right\}.
	$$
	Thus, $\mathcal{N}_k(x_i)$ is the set of the $k$ closest points to $x_i$ in $\mathbb{R}^p$.
	
	The local patch centered at $x_i$ is defined as a weighted fuzzy simplicial set
	$$
	\mathcal{P}_i
	=
	\left\{
	(x_i,x_j,w_{ij}):
	x_j\in\mathcal{N}_k(x_i),\; w_{ij}>0
	\right\}.
	$$
	Here, $x_i$ is the anchor point, $x_j\in\mathcal{N}_k(x_i)$ are the local neighbors, and $w_{ij}$ is the affinity or edge weight between $x_i$ and $x_j$. In the UMAP construction, these weights are computed from local connectivity and local scale parameters; in ISOMAP, they are replaced by normalized geodesic distances on a nearest-neighbor graph; and in DBSCAN, they are replaced by density, and cluster-based similarities.
	
	The individual patches are combined to form a global fuzzy simplicial set, or weighted graph,
	$$
	G
	=
	\bigcup_{i=1}^{n}\mathcal{P}_i.
	$$
	This graph should be interpreted as a discrete atlas of overlapping local neighborhoods. Through the one-to-one relationship between the cover of local patches and its nerve, as formalized by the Nerve Theorem and its consequences, this construction provides a discrete topological and geometric representation of the underlying data structure; see \cite{Rodriguez2024Palermo} and  \cite{rodriguez2025}. In this sense, the original Euclidean data table is transformed into a data table endowed with a discrete Riemannian-like structure.
	
	Consequently, the similarity matrix
	$
	S=(S_{ij})_{n\times n}
	$
	acts as a discrete local metric object. Large values of $S_{ij}$ indicate strong local affinity, while small values indicate weak local affinity or separation in the induced geometry. We therefore define
	$
	\rho_{ij}=1-S_{ij},
	$
	and use this quantity to define the Riemannian subtraction
	$
	x_i\ominus x_j
	=
	\rho_{ij}(x_i-x_j).
	$
	Thus, the role played by tangent vectors in a smooth manifold is replaced, in a discrete data table, by locally weighted Euclidean differences. In particular, if $g=x_\lambda$ is the discrete Riemannian mean, then the centered differences used by Riemannian Regression are
	$$
	x_i\ominus g
	=
	\rho_{i\lambda}(x_i-x_\lambda),
	\qquad
	y_i\ominus y_\lambda
	=
	\rho_{i\lambda}(y_i-y_\lambda).
	$$
	This is the point at which the Euclidean table is no longer treated as a purely Euclidean object: it is treated as a finite sample equipped with local neighborhoods, local weights, and Riemannian-type centered differences.
	
	It is also important to distinguish the present construction from PGA \cite{flet,pennec}. PGA assumes that the manifold $M$ and its geodesic distance, tangent spaces, and exponential or logarithmic maps are already available. Riemannian Regression starts from a standard data table and constructs a discrete local geometry before fitting a regression model. Hence, the proposed method is not a replacement for PGA on known manifolds; rather, it is a data-table analogue for situations in which only observations, a base dissimilarity, and local neighborhood information are available.
	
	Finally, once the similarity matrix $S$ and the center $g=x_\lambda$ are fixed, the regression estimator derived in this paper is algebraically a weighted least-squares estimator. This observation is intentional rather than problematic: the contribution is not that weighted least squares is new, but that the weights $\rho_{i\lambda}$ are obtained from a data-induced local geometry and used to define a coherent Riemannian centering of both predictors and response. Moreover, the construction is not tied to the UMAP embedding algorithm or its low-dimensional optimization step. UMAP \cite{mcinnes}, ISOMAP \cite{isomap}, and DBSCAN \cite{dbscan} are used here only as alternative mechanisms for constructing the local similarity matrix $S$. This makes the framework independent of any single embedding method and allows the analyst to compare different local geometries within the same regression model.

	\section{Local Similarities and Riemannian Differences}
	
	Let
	$$
	X=
	\begin{pmatrix}
		x_1^t\\
		x_2^t\\
		\vdots\\
		x_n^t
	\end{pmatrix}
	\in \R^{n\times p},
	\qquad
	x_i=(x_{i1},\ldots,x_{ip})\in\R^p,
	$$
	and let
	$
	y=(y_1,\ldots,y_n)^t\in\R^n
	$
	be the response variable.
	
	Suppose that a local-geometry algorithm produces a symmetric similarity matrix
	$
	S=(S_{ij})_{n\times n},
	\;
	0\leq S_{ij}\leq 1.
	$
	The larger $S_{ij}$ is, the more similar $x_i$ and $x_j$ are under the chosen local geometry. We define the associated Riemannian dissimilarity coefficient as
	$
	\rho_{ij}=1-S_{ij}.
	$
	Therefore, $\rho_{ij}\approx 0$ for very similar observations and $\rho_{ij}\approx 1$ for weakly related observations.
	
	\begin{definition}[Riemannian subtraction]
		Given $\rho_{ij}=1-S_{ij}$, the Riemannian subtraction induced by $S$ is defined by
		$
		x_i\ominus x_j=\rho_{ij}(x_i-x_j).
		$
		The associated induced distance is
		$
		d_R(x_i,x_j)=\|x_i\ominus x_j\|=\rho_{ij}\|x_i-x_j\|.
		$
	\end{definition}
	
	The diagonal convention used throughout the paper is
	$
	S_{ii}=0,
	\;
	\rho_{ii}=1.
	$
	This convention is harmless, because
	$
	x_i\ominus x_i=\rho_{ii}(x_i-x_i)=0.
	$
	It also matches the convention used in the computational implementations of the Riemannian framework.

	\section{Construction of the Similarity Matrix}
	\label{sec:similarities}
	
	The previous section explains why a similarity matrix can be interpreted as the object that equips the original Euclidean data table with a discrete local Riemannian-like geometry. We now describe how this matrix is constructed in practice. The proposed Riemannian regression method depends on the construction of a similarity matrix
	$
	S=(S_{ij})_{n\times n},
	$
	from which the dissimilarity coefficients are obtained through the common rule
	$
	\rho_{ij}=1-S_{ij}.
	$
	The purpose of this section is to describe in detail three possible ways to generate the local geometry of the data: UMAP, ISOMAP, and DBSCAN. Although the three constructions are different, they all produce a matrix $S$ with entries in $[0,1]$, and therefore they all induce the same Riemannian subtraction
	$
	x_i\ominus x_j=\rho_{ij}(x_i-x_j).
	$

	\subsection{UMAP-Based Local Similarities}
	\label{subsec:umap_similarity}
	
	UMAP \cite{mcinnes} constructs a fuzzy local neighborhood structure from the observations. Let
	$
	d_{ij}^{(0)}=d(x_i,x_j)
	$
	denote a base distance, typically the Euclidean distance. For each observation $x_i$, let
	$
	\mathcal{N}_k(x_i)=\{x_{i_1},\ldots,x_{i_k}\}
	$
	be the set of its $k$ nearest neighbors.
	
	The first local parameter is the local connectivity radius. To avoid confusion with the dissimilarity coefficient $\rho_{ij}$, we denote this UMAP local connectivity parameter by 
	$
	r_i
	=
	\min\{d(x_i,x_j):x_j\in\mathcal{N}_k(x_i),\ d(x_i,x_j)>0\}.
	$
	This guarantees that each point is strongly connected to at least one non-identical neighbor.
	
	The second local parameter is a scale $\sigma_i>0$, chosen so that the local fuzzy neighborhood around $x_i$ has approximately constant cardinality. It is computed by solving
	$$
	\sum_{x_j\in\mathcal{N}_k(x_i)}
	\exp\left(
	-
	\frac{\max\{0,d(x_i,x_j)-r_i\}}{\sigma_i}
	\right)
	=
	\log_2(k).
	$$
	Given $r_i$ and $\sigma_i$, UMAP defines the directed membership strength from $x_i$ to $x_j$ as
	$$
	a_{ij}
	=
	\begin{cases}
		\displaystyle
		\exp\left(
		-
		\frac{\max\{0,d(x_i,x_j)-r_i\}}{\sigma_i}
		\right),
		& x_j\in\mathcal{N}_k(x_i),\\[8pt]
		0,& x_j\notin\mathcal{N}_k(x_i).
	\end{cases}
	$$
	Thus, $A=(a_{ij})$ is generally a directed matrix of local affinities.
	
	To obtain a symmetric similarity matrix, UMAP applies the fuzzy union rule
	$
	S_{ij}^{\textrm{\tiny UMAP}}
	=
	a_{ij}+a_{ji}-a_{ij}a_{ji}.
	$
	Equivalently, in matrix notation,
	$
	S^{\textrm{\tiny UMAP}}
	=
	A+A^t-A\circ A^t,
	$
	where $\circ$ denotes the Hadamard product. The value $S_{ij}^{\textrm{\tiny UMAP}}$ lies in $[0,1]$: values close to $1$ indicate strong local affinity, whereas values close to $0$ indicate weak local affinity or no local connection.
	
	As in the Riemannian framework used here, we impose the diagonal convention
	$
	S_{ii}^{\textrm{\tiny UMAP}}=0.
	$
	Then the UMAP-induced dissimilarity coefficient is
	$
	\rho_{ij}^{\textrm{\tiny UMAP}}
	=
	1-S_{ij}^{\textrm{\tiny UMAP}},
	$
	and the corresponding Riemannian subtraction is
	$
	x_i\ominus x_j
	=
	\rho_{ij}^{\textrm{\tiny UMAP}}(x_i-x_j).
	$
	Therefore, points that are strongly connected in the UMAP fuzzy graph produce small Riemannian differences, while weakly connected points keep larger differences.
	
	\subsection{ISOMAP-Based Local Similarities}
	\label{subsec:isomap_similarity}
	
	The central idea is to use approximate geodesic distances computed on a nearest-neighbor graph and then transform those distances into a similarity matrix compatible with the same rule used above
	$
	\rho_{ij}=1-S_{ij}.
	$
	
	First, we compute the usual Euclidean distance matrix
	$$
	D^{(0)}=(d_{ij}^{(0)})_{n\times n},
	\qquad
	d_{ij}^{(0)}=\|x_i-x_j\|
	=
	\sqrt{\sum_{\ell=1}^{p}(x_{i\ell}-x_{j\ell})^2}.
	$$
	From this distance matrix, ISOMAP \cite{isomap} constructs a weighted $k$-nearest-neighbor graph
	$
	G_k=(V,E_k),
	\qquad
	V=\{x_1,\ldots,x_n\}.
	$
	For each point $x_i$, the set $\mathcal{N}_k(x_i)$ contains its $k$ nearest neighbors. The initial graph weight matrix can be written as
	$$
	A_{ij}^{(k)}
	=
	\begin{cases}
		d_{ij}^{(0)},&\text{if }x_j\in\mathcal{N}_k(x_i)\text{ or }x_i\in\mathcal{N}_k(x_j),\\
		+\infty,&\text{if there is no edge between }x_i\text{ and }x_j,\\
		0,&\text{if }i=j.
	\end{cases}
	$$
	
	In order to compute geodesic distances between all pairs of observations, the graph must be connected. Therefore, the algorithm starts with
	$
	k=
	$ number of nearest neighbors, 
	and if the graph is fragmented, increases $k$ successively adding one to $k$
	until the first value $k_*$ for which $G_{k_*}$ is connected is found. This value is stored as
	$k_*$.
	
	Once a connected graph is obtained, ISOMAP replaces direct Euclidean distances by approximate geodesic distances over the graph. For each pair of observations $x_i$ and $x_j$, we define
	$$
	d_{ij}^{\textrm{\tiny ISOMAP}}
	=
	\min_{\gamma:i\leadsto j}
	\sum_{(a,b)\in\gamma}d_{ab}^{(0)},
	$$
	where $\gamma:i\leadsto j$ denotes a path in the graph connecting $x_i$ and $x_j$. This produces the matrix
	$
	D^{\textrm{\tiny ISOMAP}}
	=
	\left(d_{ij}^{\textrm{\tiny ISOMAP}}\right)_{n\times n}.
	$
	
	Let be
	$
	d_{\max}
	=
	\max_{1\leq i,j\leq n}d_{ij}^{\textrm{\tiny ISOMAP}}.
	$
	The normalized ISOMAP distances are
	$$
	\widetilde d_{ij}^{\textrm{\tiny ISOMAP}}
	=
	\frac{d_{ij}^{\textrm{\tiny ISOMAP}}}{d_{\max}},
	\qquad
	0\leq \widetilde d_{ij}^{\textrm{\tiny ISOMAP}}\leq 1.
	$$
	To keep the same computational pattern as UMAP, we define an ISOMAP-induced similarity by
	$
	S_{ij}^{\textrm{\tiny ISOMAP}}
	=
	1-\widetilde d_{ij}^{\textrm{\tiny ISOMAP}}
	=
	1-
	\frac{d_{ij}^{\textrm{\tiny ISOMAP}}}{d_{\max}},
	\; i\neq j.
	$
	If two points are close in the ISOMAP geodesic geometry, then
	$
	d_{ij}^{\textrm{\tiny ISOMAP}}\approx 0
	\qquad\Longrightarrow
	S_{ij}^{\textrm{\tiny ISOMAP}}\approx 1.
	$
	If two points are far apart on the geodesic graph, then
	$
	d_{ij}^{\textrm{\tiny ISOMAP}}\approx d_{\max}$ and 
	$S_{ij}^{\textrm{\tiny ISOMAP}}\approx 0.
	$
	Finally, to match the same convention used in the UMAP-based construction, we set
	$
	S_{ii}^{\textrm{\tiny ISOMAP}}=0,
	\; i=1,\ldots,n.
	$
	Therefore, for $i\neq j$,
	$
	\rho_{ij}^{\textrm{\tiny ISOMAP}}
	=
	1-S_{ij}^{\textrm{\tiny ISOMAP}}
	=
	\frac{d_{ij}^{\textrm{\tiny ISOMAP}}}{d_{\max}},
	$
	whereas on the diagonal
	$
	\rho_{ii}^{\textrm{\tiny ISOMAP}}=1.
	$
	This diagonal convention does not create any problem, since
	$
	x_i\ominus x_i
	=
	\rho_{ii}(x_i-x_i)=0.
	$
	Thus, ISOMAP preserves the interpretation of $\rho_{ij}$ as a normalized approximate geodesic distance.
	
	\subsection{DBSCAN-Based Local Similarities}
	\label{subsec:dbscan_similarity}
	
	Unlike UMAP and ISOMAP, DBSCAN \cite{dbscan} does not directly return a continuous density. However, its definition is based on counting how many observations lie inside a ball of radius $\varepsilon$. This allows us to define a local density from the cardinality of an $\varepsilon$-neighborhood.
	
	DBSCAN depends mainly on two parameters
	$
	\varepsilon$
	and
	\textsf{MinPts}.
	The $\varepsilon$-neighborhood of $x_i$ is
	$
	N_\varepsilon(x_i)
	=
	\{x_j:\|x_i-x_j\|\leq\varepsilon\}.
	$
	A point $x_i$ is a core point if
	$
	|N_\varepsilon(x_i)|\geq \textsf{MinPts}.
	$
	A border point is not a core point but is reachable from a core point. A noise point is neither a core point nor reachable from any core point. DBSCAN assigns cluster labels
	$
	c_i\in\{0,1,2,\ldots,K\},
	$
	where
	$
	c_i=0
	$
	denotes noise, and $c_i=r\geq 1$ means that $x_i$ belongs to cluster $r$.
	
	If the user does not provide $\varepsilon$, the algorithm chooses it automatically using the distances to the $\textsf{MinPts}$-th nearest neighbor. Let
	$
	d_{i,(\textsf{MinPts})}
	$
	be the distance from $x_i$ to its $\textsf{MinPts}$-th nearest neighbor. Then
	$$
	\varepsilon
	=
	Q_q\left(
	d_{1,(\textsf{MinPts})},
	d_{2,(\textsf{MinPts})},
	\ldots,
	d_{n,(\textsf{MinPts})}
	\right),
	$$
	where $Q_q$ denotes the quantile of order $q$ of the set of nearest-neighbor distances, with $0<q<1$.
	For example, if $q=0.90$, then $\varepsilon$ is the 90th percentile of the distances
	$
	d_{1,(\textsf{MinPts})},
	d_{2,(\textsf{MinPts})},
	\ldots,
	d_{n,(\textsf{MinPts})}.
	$
	In other words, $\varepsilon$ is chosen so that approximately 90\% of these $\textsf{MinPts}$-nearest-neighbor distances are less than or equal to $\varepsilon$.

	The local density of $x_i$ is defined as the number of points in its $\varepsilon$-neighborhood
	$
	\eta_i=|N_\varepsilon(x_i)|.
	$
	Equivalently, if the point $x_i$ is excluded first, we can write
	$
	\eta_i
	=
	1+
	\left|
	\{x_j:j\neq i,\ \|x_i-x_j\|\leq\varepsilon\}
	\right|.
	$
	The normalized density is
	$
	\delta_i
	=
	\frac{\eta_i}{\max_{1\leq r\leq n}\eta_r},
	\;
	0<\delta_i\leq 1.
	$
	If $x_i$ belongs to a dense region, then $\delta_i\approx 1$; if $x_i$ belongs to a sparse region, then $\delta_i$ is smaller.
	
	The first component of the DBSCAN similarity is a distance-based similarity. Let $\sigma>0$ be a scale parameter. If it is not provided, it is taken as the median of the nonzero pairwise distances
	$
	\sigma=
	\operatorname{median}\{d_{ij}:i<j,\ d_{ij}>0\}.
	$
	Then
	$$
	S_{ij}^{\textrm{\tiny dist}}
	=
	\exp\left(-\frac{\|x_i-x_j\|}{\sigma}\right).
	$$
	This quantity is close to $1$ for nearby points and close to $0$ for distant points.
	
	The second component introduces local density information. For each pair $(i,j)$, we define
	$
	F_{ij}^{\textrm{\tiny dens}}
	=
	\left(\sqrt{\delta_i\delta_j}\right)^\alpha,
	$
	where
	$
	\alpha
	$ is a parameter call the {\em density power}. 
	The geometric mean $\sqrt{\delta_i\delta_j}$ makes the density contribution large only when both points are in relatively dense regions. If $\alpha>1$, sparse regions are penalized more strongly; if $0<\alpha<1$, the density effect is smoother.
	
	The third component uses the cluster labels produced by DBSCAN. We define
	$$
	F_{ij}^{\textrm{\tiny clust}}
	=
	\begin{cases}
		1,& \text{if }c_i=c_j\text{ and }c_i,c_j>0,\\[4pt]
		\gamma,& \text{if }c_i\neq c_j\text{ and }c_i,c_j>0,\\[4pt]
		\tau,& \text{if }c_i=0\text{ or }c_j=0,
	\end{cases}
	$$
	where
	$
	\gamma
	$ is a parameter call {\em between cluster factor},
	and  $\tau$ is a parameter call {\em noise factor}.
	Usually $0<\gamma<1$ and $0<\tau<1$, so that pairs in different clusters or pairs involving noise are penalized.
	
	Combining the distance, density, and cluster components, the DBSCAN induced similarity is
	$
	S_{ij}^{\textrm{\tiny DBSCAN}}
	=
	S_{ij}^{\textrm{\tiny dist}}
	F_{ij}^{\textrm{\tiny dens}}
	F_{ij}^{\textrm{\tiny clust}},
	$
	or, explicitly,
	$$
	S_{ij}^{\textrm{\tiny DBSCAN}}
	=
	\exp\left(-\frac{\|x_i-x_j\|}{\sigma}\right)
	\left(\sqrt{\delta_i\delta_j}\right)^\alpha
	F_{ij}^{\textrm{\tiny clust}},
	\qquad i\neq j.
	$$
	Equivalently,
	$$
	S_{ij}^{\textrm{\tiny DBSCAN}}
	=
	\exp\left(-\frac{\|x_i-x_j\|}{\sigma}\right)
	\left(
	\sqrt{
		\frac{\eta_i}{\eta_{\max}}
		\frac{\eta_j}{\eta_{\max}}
	}
	\right)^\alpha
	F_{ij}^{\textrm{\tiny clust}},
	$$
	where
	$
	\eta_{\max}=\max_{1\leq r\leq n}\eta_r.
	$
	
	Optionally, we may force similarities to exist only between points connected through the $\varepsilon$-neighborhood graph. We define
	$$
	a_{ij}^{\varepsilon}
	=
	\begin{cases}
		1,&\text{if }x_j\in N_\varepsilon(x_i)\text{ or }x_i\in N_\varepsilon(x_j),\\
		0,&\text{otherwise}.
	\end{cases}
	$$
	If this option is used, the similarity is modified as
	$
	S_{ij}^{\textrm{\tiny DBSCAN}}
	\leftarrow
	S_{ij}^{\textrm{\tiny DBSCAN}}a_{ij}^{\varepsilon}.
	$
	Finally, as in the other constructions, we impose
	$
	S_{ii}^{\textrm{\tiny DBSCAN}}=0,
	\;
	\rho_{ij}^{\textrm{\tiny DBSCAN}}
	=1-S_{ij}^{\textrm{\tiny DBSCAN}}.
	$
	Thus, if two points are close, lie in dense regions, and belong to the same non-noise cluster, then
	$
	S_{ij}^{\textrm{\tiny DBSCAN}}\approx 1,
	\qquad
	\rho_{ij}^{\textrm{\tiny DBSCAN}}\approx 0,
	$
	so their Riemannian difference is strongly contracted
	$
	x_i\ominus x_j
	=
	\rho_{ij}^{\textrm{\tiny DBSCAN}}(x_i-x_j)
	\approx 0.
	$
	On the other hand, if the points are far apart, sparse, in different clusters, or if one of them is noise, then
	$
	S_{ij}^{\textrm{\tiny DBSCAN}}\approx 0,
	\;
	\rho_{ij}^{\textrm{\tiny DBSCAN}}\approx 1,
	$
	and the Riemannian difference is close to the usual Euclidean difference.
	
	\section{Riemannian Regression}
	
	Given a similarity matrix $S$ obtained by UMAP, ISOMAP, or DBSCAN, we define
	$
	\rho_{ij}=1-S_{ij}.
	$
	The Riemannian center is selected as the discrete Fr\'{e}chet mean
	$$
	g=x_\lambda
	=
	\arg\min_{x_m\in X}
	\sum_{i=1}^n d_R(x_i,x_m)^2,
	$$
	where
	$
	d_R(x_i,x_m)=\rho_{im}\|x_i-x_m\|.
	$
	The corresponding response center is
	$
	y_\lambda=y(x_\lambda).
	$
	We define the classically centered matrices with respect to the Riemannian center:
	$$
	X_{c,\lambda}=
	\begin{pmatrix}
		x_1-x_\lambda\\
		x_2-x_\lambda\\
		\vdots\\
		x_n-x_\lambda
	\end{pmatrix},
	\qquad
	y_{c,\lambda}=
	\begin{pmatrix}
		y_1-y_\lambda\\
		y_2-y_\lambda\\
		\vdots\\
		y_n-y_\lambda
	\end{pmatrix}.
	$$
	Let be
	$
	W=\operatorname{diag}(\rho_{1\lambda},\rho_{2\lambda},\ldots,\rho_{n\lambda}).
	$
	The Riemannian-centered predictor matrix and response vector are
	$
	X_R=W X_{c,\lambda},
	\;
	y_R=W y_{c,\lambda}.
	$
	Equivalently, the $i$-th Riemannian-centered row is
	$
	x_i\ominus g=\rho_{i\lambda}(x_i-x_\lambda),
	$
	and the Riemannian-centered response is
	$
	y_i\ominus y_\lambda=\rho_{i\lambda}(y_i-y_\lambda).
	$
	
	\begin{theorem}[Riemannian least-squares estimator]
		Let $g=x_\lambda$ be the Riemannian center and let $W=\operatorname{diag}(\rho_{1\lambda},\ldots,\rho_{n\lambda})$. The Riemannian regression estimator is
		$$
		\widehat\beta_R
		=
		(X_{c,\lambda}^{t}W^2X_{c,\lambda})^{-1}
		X_{c,\lambda}^{t}W^2y_{c,\lambda},
		$$
		provided that $X_{c,\lambda}^{t}W^2X_{c,\lambda}$ is nonsingular. Equivalently,
		$$
		\widehat\beta_R=(X_R^tX_R)^{-1}X_R^t y_R.
		$$
	\end{theorem}
	
	\begin{proof}
		Riemannian regression minimizes the weighted objective
		$$
		Q(\beta)=
		\sum_{i=1}^n
		\rho_{i\lambda}^2
		\left[
		(y_i-y_\lambda)-(x_i-x_\lambda)^t\beta
		\right]^2.
		$$
		Using $X_R=W X_{c,\lambda}$ and $y_R=W y_{c,\lambda}$, this can be written as
		$$
		Q(\beta)=\|y_R-X_R\beta\|^2.
		$$
		Differentiating and equating to zero gives
		$
		-2X_R^t(y_R-X_R\widehat\beta_R)=0,
		$
		we have
		$
		X_R^tX_R\widehat\beta_R=X_R^ty_R.
		$
		Therefore,
		$$
		\widehat\beta_R=(X_R^tX_R)^{-1}X_R^ty_R.
		$$
		Substituting $X_R=W X_{c,\lambda}$ and $y_R=W y_{c,\lambda}$, and using $W^tW=W^2$, yields
		$$
		\widehat\beta_R
		=
		(X_{c,\lambda}^{t}W^2X_{c,\lambda})^{-1}
		X_{c,\lambda}^{t}W^2y_{c,\lambda}.
		$$
	\end{proof}
	
	The fitted values on the transformed Riemannian scale are
	$
	\widehat y_R=X_R\widehat\beta_R.
	$
	On the original response scale, the fitted value for $x_i$ is
	$
	\widehat y_i
	=
	y_\lambda+(x_i-x_\lambda)^t\widehat\beta_R.
	$
	The Riemannian coefficient of determination is defined as
	$$
	R^2_R
	=
	1-
	\frac{\sum_{i=1}^n (y_{R,i}-\widehat y_{R,i})^2}{\sum_{i=1}^n y_{R,i}^2}.
	$$
	This measures goodness of fit on the transformed Riemannian response scale. If we want to evaluate the fit on the original response scale, we may also compute
	$$
	R^2_{\textrm{orig}}
	=
	1-
	\frac{\sum_{i=1}^n (y_i-\widehat y_i)^2}{\sum_{i=1}^n (y_i-\bar y)^2}.
	$$
	
	The algorithm is presented below. This  algorithm is implemented in the R package \texttt{riemannianStats}, see \cite{riemannianStatsPkg}. 
	
	\begin{algorithm}[ht]
		\caption{Riemannian Regression with UMAP, ISOMAP, or DBSCAN}
		\begin{algorithmic}[1]
			\State \textbf{Input:} Predictor matrix $X$, response vector $y$, similarity method $m\in\{\textsf{UMAP},\textsf{ISOMAP},\textsf{DBSCAN}\}$.
			\State Compute the similarity matrix $S^{(m)}=(S_{ij}^{(m)})$ using the chosen method.
			\State Set $\rho_{ij}=1-S_{ij}^{(m)}$.
			\State Compute Riemannian distances $d_R(x_i,x_j)=\rho_{ij}\|x_i-x_j\|$.
			\State Select the Riemannian center $g=x_\lambda$ by minimizing $\sum_i d_R(x_i,x_m)^2$ over $x_m\in X$.
			\State Set $y_\lambda=y(x_\lambda)$.
			\State Form $X_{c,\lambda}$ and $y_{c,\lambda}$.
			\State Form $W=\operatorname{diag}(\rho_{1\lambda},\ldots,\rho_{n\lambda})$.
			\State Compute $X_R=W X_{c,\lambda}$ and $y_R=W y_{c,\lambda}$.
			\State Estimate $\widehat\beta_R=(X_{c,\lambda}^tW^2X_{c,\lambda})^{-1}X_{c,\lambda}^tW^2y_{c,\lambda}$.
			\State Compute fitted values and $R^2_R$.
			\State \textbf{Output:} $\widehat\beta_R$, $g=x_\lambda$, $W$, $R^2_R$, and fitted values.
		\end{algorithmic}
	\end{algorithm}
	
	\section{Illustrative Examples}
	
	The following illustrative examples are intended to show how Riemannian Regression modifies the geometry of the classical regression problem by replacing ordinary Euclidean centered differences with locally weighted Riemannian differences. All numerical results in these examples can be executed and verified using the R package \texttt{riemannianStats}, which implements Riemannian methods for principal component analysis, regression, and visualization, see \cite{riemannianStatsPkg} and \cite{rodriguez_pypi}. In particular, the package provides computational tools for constructing local similarity structures, obtaining the corresponding dissimilarity coefficients $\rho_{ij}=1-S_{ij}$, computing Riemannian centers, transforming the data into Riemannian-centered coordinates, and fitting the resulting Riemannian regression models. Therefore, the examples are not only illustrative but also reproducible within a standard R environment.

	\subsection{Example 1: Observations Difficult for Classical Regression}
	
	Consider a data table with $n=10$ individuals and $p=5$ predictor variables. In this first example, observations $x_4$ and $x_5$ are deliberately chosen as difficult observations for classical regression, because their response values are atypical relative to the general trend followed by the rest of the data. The purpose is to show how the geometry induced by the similarity matrix can substantially modify the centered representation and improve the fit.
	
	\begin{center}
		\begin{tabular}{c|ccccc|r}
			\toprule
			Cases & $X_1$ & $X_2$ & $X_3$ & $X_4$ & $X_5$ & $y\;\;\;\;$ \\
			\midrule
			$x_1$  & $0.0$ & $0.0$ & $0.1$ & $0.2$ & $0.0$ & $8.77$ \\
			$x_2$  & $0.5$ & $0.1$ & $0.2$ & $0.0$ & $0.1$ & $9.09$ \\
			$x_3$  & $0.2$ & $0.8$ & $0.1$ & $0.3$ & $0.0$ & $8.35$ \\
			$x_4$  & $0.9$ & $0.9$ & $0.2$ & $0.2$ & $0.1$ & $25.00$ \\
			$x_5$  & $1.2$ & $0.3$ & $0.3$ & $0.1$ & $0.2$ & $-10.00$ \\
			$x_6$  & $1.5$ & $0.7$ & $0.3$ & $0.3$ & $0.1$ & $10.00$ \\
			$x_7$  & $4.8$ & $5.0$ & $5.0$ & $4.7$ & $4.9$ & $12.47$ \\
			$x_8$  & $5.5$ & $5.1$ & $5.1$ & $4.8$ & $5.0$ & $12.97$ \\
			$x_9$  & $5.2$ & $5.8$ & $4.9$ & $5.3$ & $5.1$ & $12.31$ \\
			$x_{10}$ & $6.0$ & $5.6$ & $5.2$ & $5.4$ & $5.3$ & $13.33$ \\
			\bottomrule
		\end{tabular}
	\end{center}
	
	Figure~\ref{fig:example1-shape} provides a simple two-dimensional illustration of the data configuration. The specific coordinates in this drawing are only schematic, but they help visualize that most observations are aligned around a common trend, whereas $x_4$ and $x_5$ appear as response outliers. Thus, the classical Euclidean fit must accommodate two observations whose $y$-values are much larger and much smaller than the remaining responses.
	
	\begin{figure}[ht]
		\centering
		\includegraphics[width=0.92\textwidth]{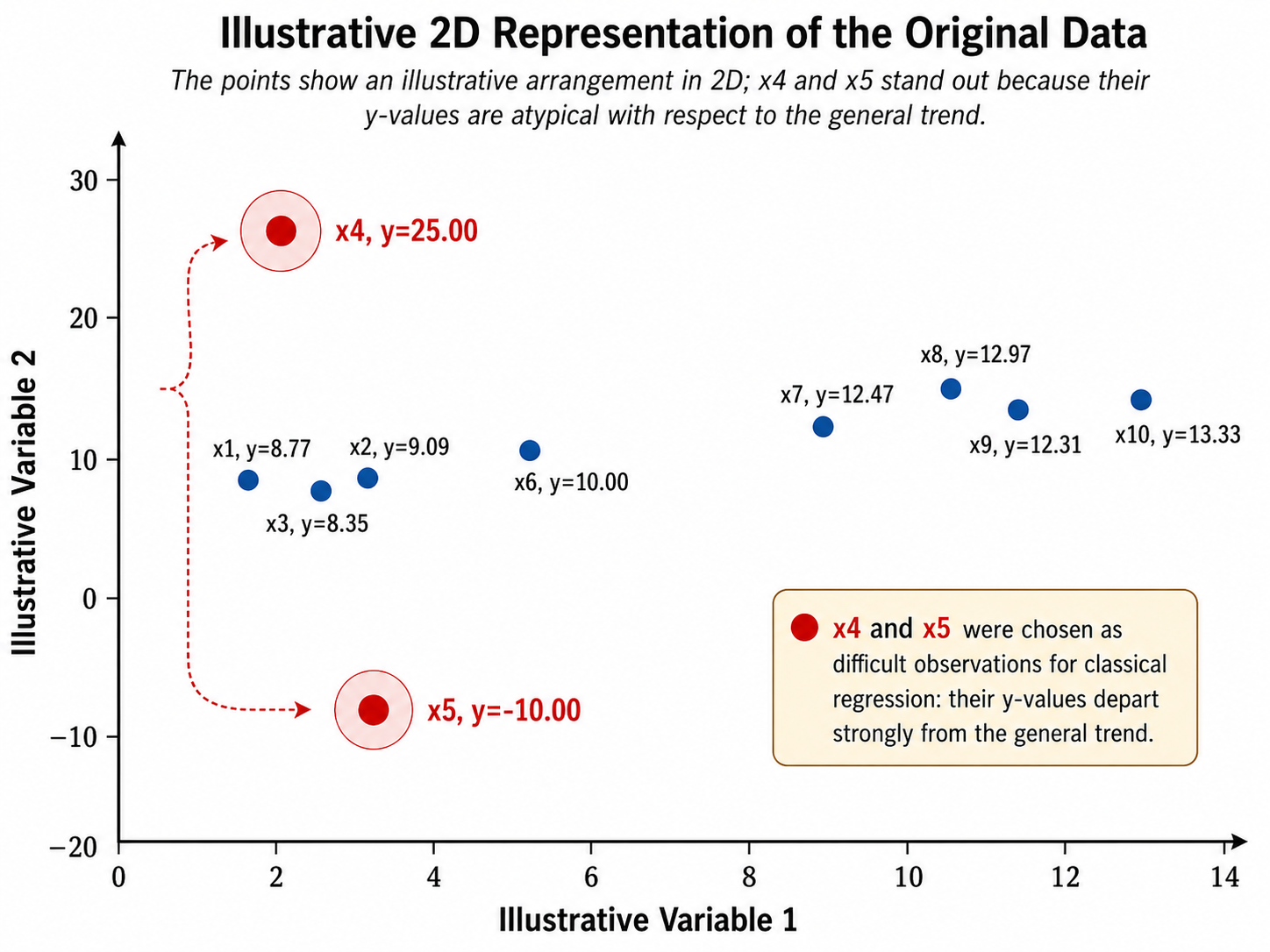}
		\caption{Illustrative two-dimensional representation for Example~1. Observations $x_4$ and $x_5$ stand out because their response values are atypical with respect to the overall trend.}
		\label{fig:example1-shape}
	\end{figure}
	
	The classical centered regression uses the Euclidean centroid
	$$
	\bar{x}=(2.58,2.43,2.14,2.13,2.08),
	\qquad
	\bar y=10.229,
	$$
	so that all centered differences are computed in the usual unweighted way. In that representation, the classical regression gives
	$$
	R^2_{\textrm{classical}}\approx 0.3665.
	$$
	Hence, only about $36.65\%$ of the variation is explained on the classical centered scale.
	
	Now consider the Riemannian version induced by UMAP similarities \cite{mcinnes}. The Riemannian center is selected by minimizing the weighted squared distances, and in this example the minimum is attained at
	$$
	g=x_\lambda=x_6=(1.5,0.7,0.3,0.3,0.1),
	\qquad
	y_\lambda=10.00.
	$$
	The corresponding Riemannian centering weights are approximately
	$$
	(\rho_{16},\rho_{26},\rho_{36},\rho_{46},\rho_{56},\rho_{66},\rho_{76},\rho_{86},\rho_{96},\rho_{10;6})
	$$
	$$
	\approx
	(1.000,0.977,1.000,0.000,0.000,1.000,1.000,1.000,1.000,1.000).
	$$
	These weights are the key feature of the example. Because
	$
	\rho_{46}=0,
	\;
	\rho_{56}=0,
	$
	the Riemannian centered differences satisfy
	$
	x_4\ominus g=0,
	\;
	x_5\ominus g=0.
	$
	Therefore, the two problematic observations are collapsed at the Riemannian center and no longer dominate the fit in the transformed geometry.
	
	The resulting Riemannian regression gives
	$$
	R^2_R\approx 0.9988.
	$$
	Thus, on the Riemannian centered scale, the model explains approximately $99.88\%$ of the variation. Figure~\ref{fig:example1-compare} compares the classical centered regression and the UMAP-induced Riemannian regression. The left panel shows the poor classical fit, while the right panel shows that the transformed observations lie almost exactly on the identity line.
	
	\begin{figure}[ht]
		\centering
		\includegraphics[width=\textwidth]{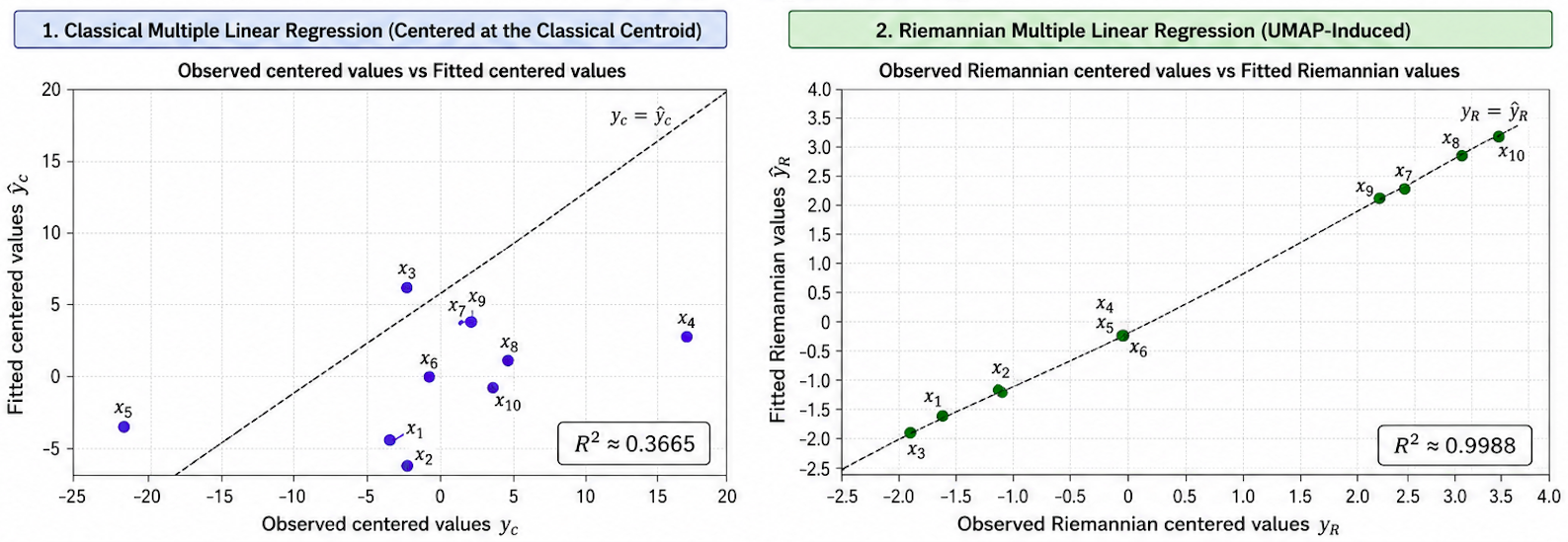}
		\caption{Comparison between classical multiple linear regression and UMAP-induced Riemannian multiple linear regression for Example~1. The atypical observations $x_4$ and $x_5$ heavily affect the classical fit, whereas the Riemannian centering makes the relationship almost perfectly linear on the transformed scale.}
		\label{fig:example1-compare}
	\end{figure}
	
	This example shows very clearly that the proposed method does not modify the algebraic linear structure of the regression model. Instead, it changes the geometry through the similarity matrix, the dissimilarity coefficients $\rho_{ij}$, and the weighted centering around the Riemannian center.
	
	\subsection{Example 2: Two Clusters with Different Densities}
	
	Consider again a data table with $n=10$ individuals and $p=5$ predictor variables. In this second example there are no artificially introduced response outliers. Instead, the data are organized into two groups with markedly different local densities: a dense cluster $A$ with six observations and a less dense cluster $B$ with four observations. The purpose is to illustrate how Riemannian Regression reacts to heterogeneity in local density.
	
	\begin{center}
		\begin{tabular}{c|c|rrrrr|c}
			\toprule
			Cases & Cluster & $X_1$ & $X_2$ & $X_3$ & $X_4$ & $X_5$ & $y$ \\
			\midrule
			$x_1$  & A & $0.00$ & $0.02$ & $-0.01$ & $0.03$ & $0.00$ & $4.286$ \\
			$x_2$  & A & $0.04$ & $-0.01$ & $0.02$ & $0.01$ & $-0.02$ & $5.486$ \\
			$x_3$  & A & $-0.03$ & $0.01$ & $0.00$ & $-0.02$ & $0.02$ & $4.776$ \\
			$x_4$  & A & $0.02$ & $0.04$ & $-0.02$ & $0.00$ & $0.01$ & $4.762$ \\
			$x_5$  & A & $-0.01$ & $-0.03$ & $0.01$ & $0.02$ & $-0.01$ & $5.015$ \\
			$x_6$  & A & $0.03$ & $0.00$ & $0.03$ & $-0.01$ & $0.00$ & $5.289$ \\
			$x_7$  & B & $3.00$ & $3.50$ & $2.80$ & $3.20$ & $3.10$ & $4.500$ \\
			$x_8$  & B & $4.00$ & $2.90$ & $3.40$ & $3.80$ & $2.70$ & $3.698$ \\
			$x_9$  & B & $2.70$ & $4.10$ & $3.60$ & $2.90$ & $3.90$ & $3.718$ \\
			$x_{10}$ & B & $4.50$ & $4.20$ & $2.90$ & $4.30$ & $3.50$ & $4.655$ \\
			\bottomrule
		\end{tabular}
	\end{center}
	
	The geometric structure of the data is shown schematically in Figure~\ref{fig:example2-shape}. Cluster $A$ forms a compact cloud of six observations, whereas cluster $B$ is formed by four observations that are farther apart from one another. In other words, the distinction between the two groups is not driven by response anomalies, but by a strong difference in local density.
	
	\begin{figure}[ht]
		\centering
		\includegraphics[width=0.92\textwidth]{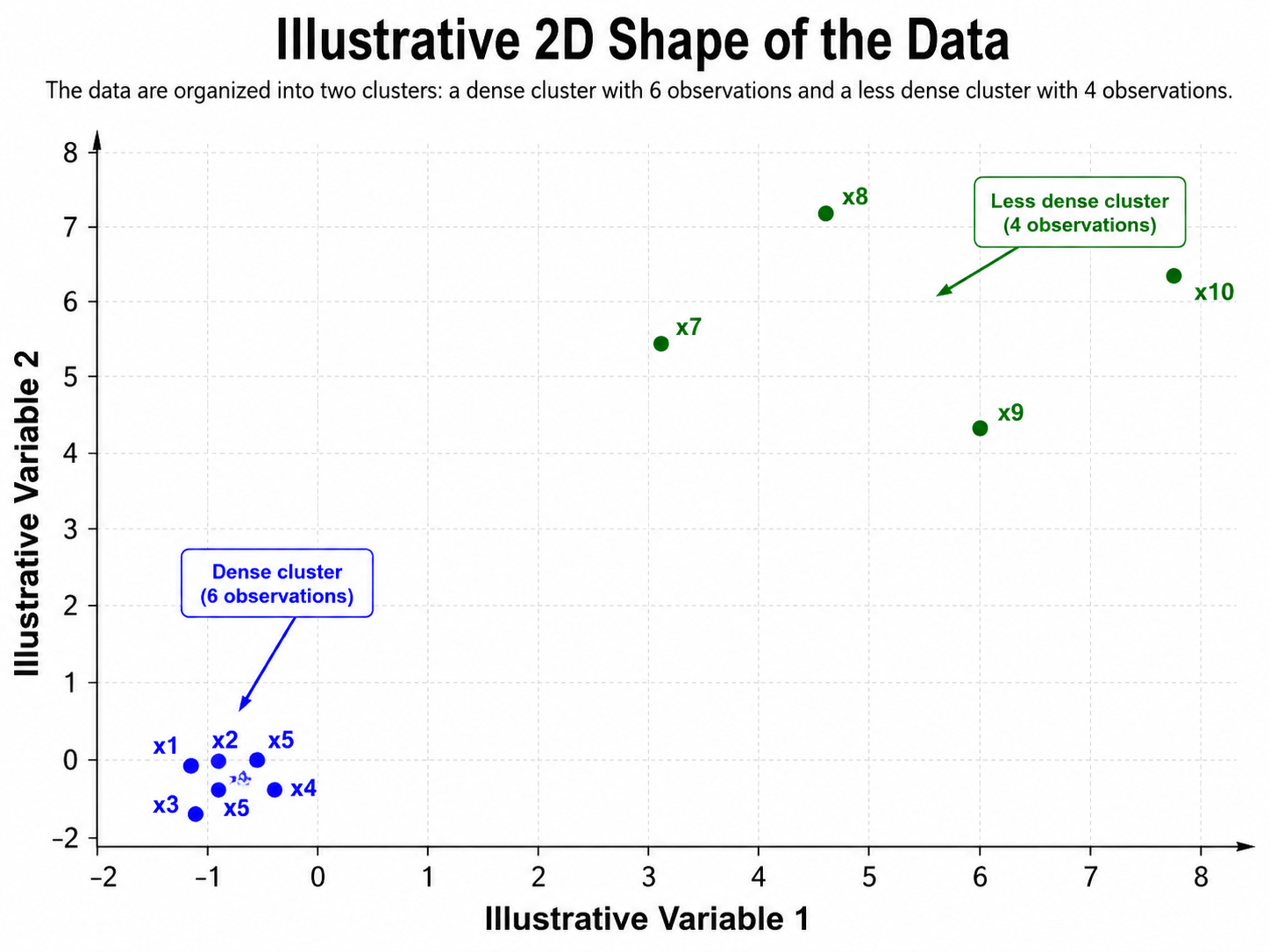}
		\caption{Illustrative two-dimensional representation for Example~2. The data are organized into a dense cluster with six observations and a less dense cluster with four observations.}
		\label{fig:example2-shape}
	\end{figure}
	
	Quantitatively, the average pairwise distance within cluster $A$ is approximately $0.065$, whereas the average pairwise distance within cluster $B$ is approximately $1.897$. Therefore, cluster $A$ is much denser than cluster $B$.
	The classical centered regression gives
	$$
	R^2_{\textrm{classical}}\approx 0.6922.
	$$
	This is already a reasonable fit, but it treats every observation with the same Euclidean centering mechanism.
	Using the local geometry induced by UMAP \cite{mcinnes}, the Riemannian center is again found at
	$
	g=x_\lambda=x_6,
	\;
	y_\lambda=y_6=5.289.
	$
	The associated weights are approximately
	$$
	(\rho_{16},\rho_{26},\rho_{36},\rho_{46},\rho_{56},\rho_{66},\rho_{76},\rho_{86},\rho_{96},\rho_{10,6})
	$$
	$$
	\approx
	(0.504,0.000,0.555,0.392,0.465,1.000,0.998,0.994,0.953,1.000).
	$$
	These values show a characteristic effect of the geometry. Since the center belongs to the dense cluster $A$, several observations in that cluster receive smaller weights, which contracts them toward the center in the transformed space. On the other hand, the observations in the sparser cluster $B$ receive weights close to $1$, so they preserve more of their original centered differences.
	The Riemannian regression then yields
	$$
	R^2_R\approx 0.9469.
	$$
	The increase with respect to the classical fit is
	$
	0.9469-0.6922=0.2547,
	$
	that is, about $25.47$ percentage points. Figure~\ref{fig:example2-compare} displays the full comparison between the classical centered regression and the UMAP-induced Riemannian regression.
	
	\begin{figure}[ht]
		\centering
		\includegraphics[width=\textwidth]{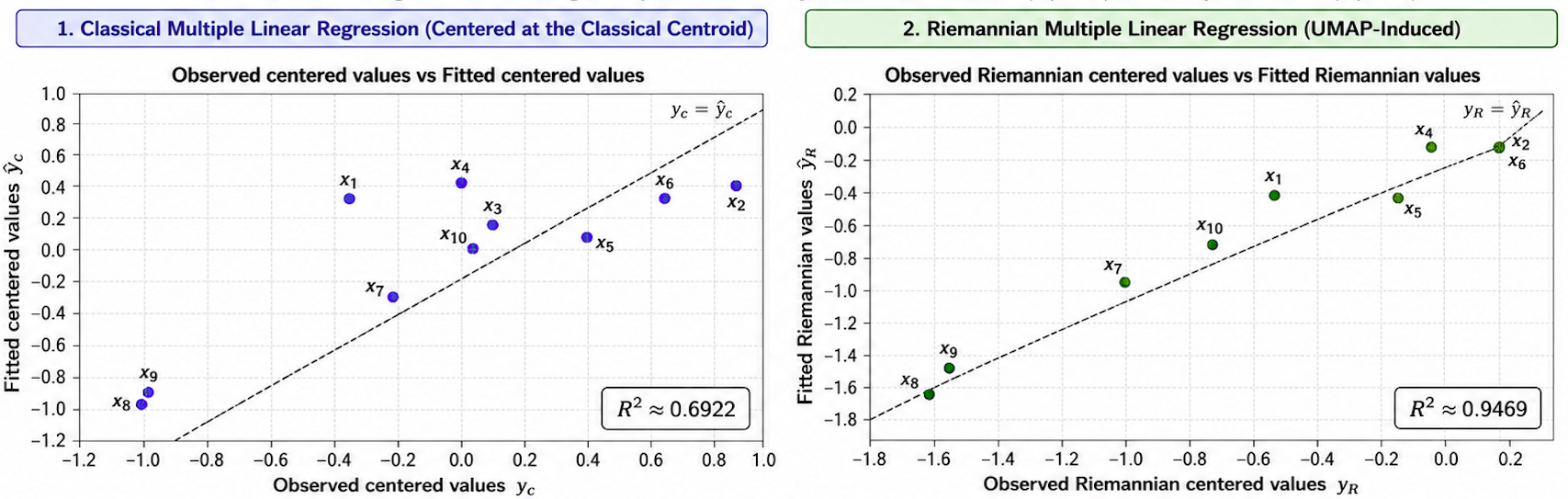}
		\caption{Comparison between classical multiple linear regression and UMAP-induced Riemannian multiple linear regression for Example~2. The main effect here is not the presence of outliers, but the difference between a dense cluster and a sparser cluster.}
		\label{fig:example2-compare}
	\end{figure}
	
	This second example shows that the proposed method can improve the fit even when there are no anomalous responses. The benefit comes from adapting the geometry of the regression problem to the local density structure of the data.
	
	\subsection{Example 3: The Abalone Data Set}
	
	The \textit{Abalone} data set \cite{abaloneUCI}, from the UCI Machine Learning Repository, contains morphometric information for $4177$ abalones. 
	Each observation corresponds to one individual and includes physical measurements such as length, diameter, height, whole weight, shucked weight, viscera weight, and shell weight. The response variable is the estimated age, computed from the number of rings as
	$
	\text{Age}=\text{Rings}+1.5.
	$
	
	This data set is useful for evaluating Riemannian Regression because the relationship between the physical measurements and age is not purely linear. Growth patterns may vary across groups, and the data may contain nonlinear local structures, density variations, and morphological variability. Therefore, a classical linear model can capture part of the global trend, but it may fail to represent local geometric effects.
	
	The classical linear regression model gives
	$
	R^2_{\textrm{lm}} \approx 0.5379.
	$
	Thus, the Euclidean linear model explains approximately $53.79\%$ of the variability in age.
	
	We compare this result with three versions of Riemannian Regression, where the local geometry is generated using UMAP, ISOMAP, and DBSCAN. The following table summarizes the values of $R^2$ obtained for different values of the neighborhood parameter, see the Figure \ref{fig:abalone_r2_comparison}.
	
	\begin{center}
		\begin{tabular}{c|cccc}
			\hline
			$n.\textsf{neighbors}$ & UMAP & ISOMAP & DBSCAN & Classical lm \\
			\hline
			$10$  & $0.7275$ & $0.8088$ & $0.7610$ & $0.5379$ \\
			$20$  & $0.7278$ & $0.8088$ & $0.7703$ & $0.5379$ \\
			$30$  & $0.7281$ & $0.8088$ & $0.7732$ & $0.5379$ \\
			$50$  & $0.7284$ & $0.8187$ & $0.6192$ & $0.5379$ \\
			$75$  & $0.7285$ & $0.5902$ & $0.7847$ & $0.5379$ \\
			$100$ & $0.7287$ & $0.5863$ & $0.7873$ & $0.5379$ \\
			$150$ & $0.7288$ & $0.5862$ & $0.7926$ & $0.5379$ \\
			\hline
		\end{tabular}
	\end{center}
	
	\begin{figure}[h]
		\centering
		\includegraphics[scale=0.30]{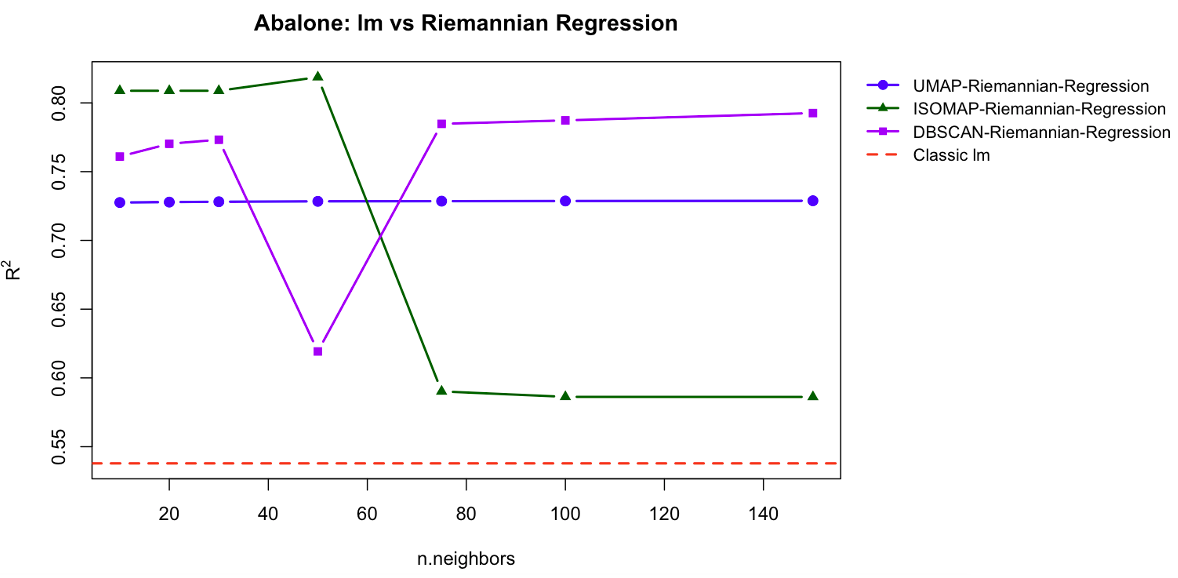}
		\caption{Comparison of the coefficient of determination $R^2$ for classical linear regression and Riemannian Regression using UMAP, ISOMAP, and DBSCAN similarities on the Abalone data set.}
		\label{fig:abalone_r2_comparison}
	\end{figure}
	
	The UMAP-based Riemannian Regression is very stable across the values of $n.\textrm{neighbors}$. Its coefficient of determination remains close to
	$
	R^2_{\textrm{UMAP}} \approx 0.728.
	$
	This represents an improvement of approximately
	$
	0.7288-0.5379 \approx 0.1909,
	$
	that is, about $19.09$ percentage points over the classical linear model. This suggests that the UMAP-induced local geometry captures a nonlinear structure that is not represented by the Euclidean regression.
	
	The ISOMAP-based Riemannian Regression gives the best result for small and moderate neighborhood sizes. In particular, for $n.\textrm{neighbors}=50$, it reaches
	$
	R^2_{\textrm{ISOMAP}} \approx 0.8187.
	$
	This is the highest value obtained in this experiment. Compared with the classical model, the improvement is
	$
	0.8187-0.5379 \approx 0.2808,
	$
	or approximately $28.08$ percentage points. This indicates that, for this range of neighborhood sizes, the graph-geodesic distances generated by ISOMAP provide a very effective local geometry for modeling the relationship between the abalone measurements and age.
	
	However, the ISOMAP results also show sensitivity to the neighborhood parameter. For larger values, such as $n.\textrm{neighbors}=75,100,150$, the coefficient of determination decreases to approximately
	$
	R^2_{\textrm{ISOMAP}}\approx 0.59.
	$
	This behavior is consistent with the nature of ISOMAP: when the neighborhood graph becomes too dense, the shortest-path distances may lose part of their geodesic character and become closer to ordinary Euclidean distances. Therefore, ISOMAP can be very powerful, but it requires a careful selection of the neighborhood size.
	
	The DBSCAN-based Riemannian Regression also improves substantially over the classical linear model. Its best value in the grid is obtained at $n.\textrm{neighbors}=150$, this is
	$
	R^2_{\textrm{DBSCAN}} \approx 0.7926.
	$
	This corresponds to an improvement of
	$
	0.7926-0.5379 \approx 0.2547,
	$
	or about $25.47$ percentage points. This suggests that the density-based similarity induced by DBSCAN captures meaningful local structure in the Abalone data set.
	
	Unlike UMAP, DBSCAN is not based only on nearest-neighbor affinities. It incorporates information about local density, cluster membership, and possible noise points. Therefore, its strong performance indicates that the Abalone data may contain regions with different local densities, and that accounting for these density variations improves the regression fit.
	
	Overall, all three Riemannian approaches outperform the classical linear model in this experiment:
	$
	R^2_{\textrm{lm}} \approx 0.5379,
	$
	whereas the Riemannian versions achieve values between approximately $0.728$ and $0.819$ for suitable parameter choices. The best performance is obtained with ISOMAP at $n.\textrm{neighbors}=50$:
	$
	R^2_{\textrm{ISOMAP}} \approx 0.8187.
	$
	
	These results show that the improvement does not come from changing the linear regression formula itself, but from changing the geometry in which the centered differences are computed. Classical regression uses the Euclidean differences
	$
	x_i-\bar{x},
	\qquad
	y_i-\bar{y},
	$
	whereas Riemannian Regression uses locally weighted differences of the form
	$
	x_i\ominus g
	=
	\rho_{i\lambda}(x_i-x_\lambda),
	\qquad
	y_i\ominus y_\lambda
	=
	\rho_{i\lambda}(y_i-y_\lambda).
	$
	
	Thus, the Abalone example provides empirical evidence that incorporating local geometry through UMAP, ISOMAP, or DBSCAN can reveal a stronger linear relationship in the transformed Riemannian coordinates than the we observed in the original Euclidean coordinates.
	
	In summary:
	\begin{itemize}
		\item UMAP gives a stable improvement across all tested neighborhood sizes.
		\item ISOMAP gives the highest value of $R^2$, but it is more sensitive to the neighborhood parameter.
		\item DBSCAN gives a strong improvement by incorporating density and cluster information.
		\item The classical linear model explains about $53.79\%$ of the variability, while the best Riemannian model explains about $81.87\%$.
	\end{itemize}

	\section{Conclusions}
	
	This paper presents Riemannian Regression, a generalization of classical multiple linear regression based on locally induced metrics. Starting from a similarity matrix $S$, the method constructs dissimilarity weights $\rho_{ij}=1-S_{ij}$ and uses them to define Riemannian differences. The regression problem becomes a weighted least-squares problem centered at a Riemannian center $g=x_\lambda$.
	
	The framework unifies three ways to define the local geometry: UMAP, ISOMAP, and DBSCAN. UMAP provides fuzzy similarities from a nearest-neighbor graph; ISOMAP provides normalized geodesic dissimilarities; DBSCAN provides density- and cluster-based similarities. Each method induces a different geometry and therefore a different regression fit.
	
	The examples show that Riemannian Regression can reduce the influence of difficult observations and adapt to regions with different local densities. Future work includes systematic comparison of the three similarity methods, sensitivity analysis with respect to their hyperparameters, inference for the Riemannian coefficients, and extensions to generalized linear models, classification, and symbolic data analysis.

\end{document}